\documentclass{amsart}
\usepackage{amsthm,amscd,amssymb,verbatim,epsf,amsmath,amsfonts,mathrsfs,graphicx}
\usepackage[colorlinks=true,linkcolor=blue,citecolor=blue]{hyperref}

\begin{document}
\theoremstyle{plain}
\newtheorem{Thm}{Theorem}
\newtheorem{Cor}{Corollary}
\newtheorem{Con}{Conjecture}
\newtheorem{Main}{Main Theorem}
\newtheorem{Lem}{Lemma}
\newtheorem{Prop}{Proposition}
\theoremstyle{definition}
\newtheorem{Def}{Definition}
\newtheorem{Note}{Note}
\newtheorem{Ex}{Example}
\theoremstyle{remark}
\newtheorem{notation}{Notation}
\renewcommand{\thenotation}{}
\errorcontextlines=0
\numberwithin{equation}{section}
\renewcommand{\rm}{\normalshape}%

\title[Mean Curvature Flow in Higher Codimension]%
   {Mean Curvature Flow of Compact Spacelike Submanifolds in Higher Codimension}
\author{Brendan Guilfoyle}
\address{Brendan Guilfoyle\\
          School of Science, Technology, Engineering and Mathematics\\
          Institute of Technology, Tralee \\
          Clash \\
          Tralee  \\
          Co. Kerry \\
          Ireland.}
\email{brendan.guilfoyle@ittralee.ie}
\author{Wilhelm Klingenberg}
\address{Wilhelm Klingenberg\\
 Department of Mathematical Sciences\\
 University of Durham\\
 Durham DH1 3LE\\
 United Kingdom}
\email{wilhelm.klingenberg@durham.ac.uk }

\begin{abstract}
We prove long-time existence for mean curvature flow of a smooth $n$-dimensional spacelike submanifold of an $n+m$ dimensional manifold whose metric satisfies the timelike curvature condition. 

\end{abstract}
\maketitle
\today
\tableofcontents

In this paper we establish the following result on long-time existence for the evolution by mean curvature flow of compact spacelike submanifolds of indefinite manifolds:

\vspace{0.1in}

\begin{Thm}\label{t:1}
Let $\Sigma_0$ be a smooth compact $n$-dimensional spacelike submanifold of an $n+m$ dimensional manifold ${\mathbb M}$ with indefinite metric ${\mathbb G}$ satisfying the timelike curvature condition (\ref{e:tcc}). 

Then there exists a unique family $f_s:\Sigma\rightarrow{{\mathbb M}}$ for $0\leq s< s_0$ of smooth compact $n$-dimensional spacelike submanifolds satisfying the initial value problem
\[
\frac{d f}{ds}=H \qquad\qquad f_0(\Sigma)=\Sigma_0,
\]
where $H$ is the mean curvature vector associated to the immersion $f_s$ in $({\mathbb M},{\mathbb G})$.

Moreover, if $f_s(\Sigma)$ remains in a smooth compact region of ${\mathbb M}$ for $0 \le s <s_0$, then $f_s$ may be extended beyond $s_0$.
\end{Thm}

\vspace{0.1in}

The critical ingredient of the proof is a gradient estimate - Proposition \ref{p:gradest} in this paper - and originally proven in the stationary case by Robert Bartnik in his 1983 thesis, see \cite{Bart}. 

Mean curvature flow of spacelike hypersurfaces in indefinite spaces has been studied previously, for example \cite{Eck} \cite{EaH} \cite{thorpe}, as has higher codimension mean curvature flow in definite spaces \cite{AandS} \cite{CaL} \cite{LaS} \cite{Sm} \cite{Wang2}.  

Here, our method is to extend the work in \cite{EaH} to higher codimension. We are generally interested in open manifolds, as there are well-known topological obstructions to the existence of indefinite metrics on compact manifolds, for example see \cite{Mats91}. It is worth noting that longtime existence in the case of codimension 1 has more recently been established without the timelike convergence condition \cite{gerhardt}. 

Mean curvature flow has found many applications, for example, probing the existence of special Lagrangian submanifolds in Calabi-Yau manifolds \cite{TaY} and of holomorphic curves in Einstein 4-manifolds \cite{CaL}. These applications arise since such submanifolds minimize area in their homology classes and therefore deforming by mean curvature flow is a natural method for finding minimizers \cite{SaW}. 

The flow  has also been used to find ``nice'' maps between two Riemannian $n$-manifolds, by flowing graphs in the product $n+n$-manifold. This has been considered both for definite and indefinite products, for example see \cite{LaL} \cite{LaS}. 

Flowing submanifolds of indefinite (rather than Riemannian) spaces can be better behaved for flowing by mean curvature.  This has been seen to be the case in the case of spacelike hypersurfaces \cite{EaH}, and now, by virtue of Theorem \ref{t:1}, in higher co-dimension. 

The motivating context of the current work is that of invariant metrics on spaces of oriented geodesics, which are often of indefinite signature \cite{agk} \cite{gg} \cite{gk}. Interestingly, special Lagrangian submanifolds in indefinite geodesic spaces have been considered from a stationary point of view recently \cite{anc}.

The result is stated as generally as possible, the specific long-time behaviour of the flow being dependent upon the particular context. The conditions introduced are mild enough to hold, for example, for indefinite product spaces, as well as warped products with a compact factor. 

The additional ingredient required for convergence would be the construction of barriers, which would depend upon more detailed information about the ambient manifold. 

In the next section we discuss a number of examples where spacelike higher co-dimension mean curvature arises. The following three sections introduce the background material, while Section \ref{s4} contains the proof of the gradient estimate. The final section contains the proof of Theorems \ref{t:1}.

\vspace{0.1in}

\section{Flows in indefinite manifolds}

\begin{Ex}[Spaces of Oriented Geodesics]

Spaces of oriented geodesics of symmetric spaces often admit canonical indefinite metrics \cite{agk}. Consider the collection ${\mathbb L}({\mathbb R}^3)$ of oriented geodesics of Euclidean 3-space, which may be identified with the total space of the tangent bundle to the 2-sphere.

This non-compact 4-manifold admits a canonical metric ${\mathbb G}_{2,2}$ of signature $(2,2)$, which, up to a spherical summand, is unique \cite{salvai}. This metric is K\"ahler, with compatible complex and symplectic structures, and is scalar flat, although it is not Einstein \cite{gk}. 

An oriented smooth surface in ${\mathbb R}^3$ gives rise, through its oriented normal lines, to a smooth surface in ${\mathbb L}({\mathbb R}^3)$. This surface is Lagrangian and the induced metric is either Lorentz or degenerate, where the degeneracy occurs precisely at the umbilic points of the surface.

Theorem \ref{t:1} arose in the context of co-dimension two mean curvature flow in ${\mathbb L}({\mathbb R}^3)$ as one element of the proof of the Carath\'eodory Conjecture on the number of umbilic points on a closed convex sphere. This involves flowing a spacelike disc with boundary lying on a Lagrangian surface and therefore requires additional boundary estimates \cite{gak}.

Spacelike surfaces in ${\mathbb L}({\mathbb R}^3)$ may also be characterized as foliations of the underlying space \cite{gs} and mean curvature flow would be a natural way of deforming such geodesic foliations. Theorem \ref{t:1} gives interior estimates for such deformations. 
\end{Ex}
\vspace{0.1in}
\begin{Ex}[Product Manifolds]
Given the indefinite product metric ${\mathbb G}=g_1-g_2$ on a product ${\mathbb M}=M_1\oplus M_2$ of $n$- and $m$-dimensional Riemannian manifolds, one can consider the mean curvature flow of an $n$-dimensional spacelike sub-manifold. 

This was carried out in \cite{LaS}, where long-time existence and convergence is established for products in which the sectional curvatures satisfy $K_2\leq K_1$. For $n=m=2$ this is equivalent to the timelike curvature condition.

\end{Ex}
\vspace{0.1in}

\begin{Ex}[A Geometric Quasi-linear Navier-Stokes Flow]

Consider the total space of the tangent bundle $T{\mathbb R}^n$ of Euclidean $n$-space, together with its natural projection $\pi:T{\mathbb R}^n\rightarrow {\mathbb R}^n$. This $2n$-manifold admits a flat metric of signature $(n,n)$ defined as follows. By definition
\[
T{\mathbb R}^n=\{(p,V)\;|\:p\in{\mathbb R}^n\qquad V\in T_p{\mathbb R}^n\}.
\]
Let $(x^1,x^2,...,x^n)$ be flat coordinates on ${\mathbb R}^n$ and for any $V\in T_p{\mathbb R}^n$ define conjugate coordinates $(\dot{x}^1,\dot{x}^2,...,\dot{x}^n)$ by
\[
V=\sum_{i=1}^n\dot{x}^i\frac{\partial}{\partial x^i}.
\]
Define the neutral metric ${\mathbb G}_{(n,n)}$ in terms of the coordinates $(x^1,x^2,...,x^n,\dot{x}^1,\dot{x}^2,...,\dot{x}^n)$ on $T{\mathbb R}^3$ by
\[
ds^2=\sum_{i=1}^ndx^id\dot{x}^i.
\]

A vector field on Euclidean 3-space is a section of the bundle $\pi:T{\mathbb R}^3\rightarrow {\mathbb R}^3$, that is, a map ${\mathcal V}:{\mathbb R}^3\rightarrow {\mathbb R}^3$ such that $\pi\circ {\mathcal V}=Id$. Denote by $\tilde{\mathbb G}$ the metric induced on ${\mathcal  V}$  by the canonical metric ${\mathbb G}_{(3,3)}$ on $T{\mathbb R}^3$.

We are interested in flowing 3-dimensional spacelike submanifolds. Examples of such can be found by considering the vector field given on ${\mathbb R}^3-(0,0,0)$ by
\[
V=H(R)\frac{\partial}{\partial R},
\]
where $R$ is the distance to the origin (the source). Such a vector field gives rise to a metric $\tilde{\mathbb G}$ that has the following signature:
\begin{center}
\begin{tabular}{ c |c c c}
  & $H'<0$ & $H'=0$ & $H'>0$\\  \hline
$H<0$& -3 & $(0,2)$ & $(1,2)$ \\ 
$H=0$& $(0,1)$ & $0$ & $(1,0)$ \\ 
$H>0$& $(2,1)$ & $(2,0)$ & +3. 
\end{tabular}
\end{center}
Thus, we get spacelike submanifolds when $H>0$ and $H'>0$.

The mean curvature vector of the embedded 3-manifold is easily computed to be
\[
\tilde{\mathbb H}=-\frac{RHH''+2RH'^2-2HH'}{\sqrt{2}2RH(H')^2}\left(\frac{\partial}{\partial R}-H'\frac{\partial}{\partial \dot{R}}\right).
\]
Co-dimension 3 mean curvature flow of these vector fields is determined by the single equation  
\[
\frac{\partial H}{\partial t}=\frac{RHH''+2RH'^2-2HH'}{2RHH'}.
\]
This is a quasi-linear Navier-Stokes equation for the vector field: a second order reaction-diffusion equation with convection and a pressure source given by the gradient of the Gauss map.

Moreover, the timelike curvature condition holds and we can apply Theorem \ref{t:1} in this setting for interior estimates.
\end{Ex}

\vspace{0.1in}

\section{Immersed spacelike submanifolds}\label{s:setting}

Let ${\mathbb M}$ be an $n+m-$dimensional manifold endowed with a metric ${\mathbb G}$ of signature ($n,m$). Throughout we use the summation convention on repeated indices, except for the quantity $\psi_\alpha$, defined below. In some instances we include summation signs for clarity. Note that raising and lowering normal indices (Greek indices) changes the sign of the component, while raising and lowering tangent indices (Latin indices) does not change the sign. For notational convenience we will use $<\cdot,\cdot>$ interchangeably with ${\mathbb G}(\cdot,\cdot)$.

We will use throughout a {\it multi-time function} $t:{\mathbb M}\rightarrow {\mathbb R}^m$ of maximal rank with
components $t_\alpha$ for $\alpha=1,...,m$ such that
\[
{\mathbb G}(\overline{\nabla}t_\alpha,\overline{\nabla}t_\alpha)<0 \qquad\qquad \forall \alpha=1,...,m,
\]
and $\{\overline{\nabla}t_\alpha\}_1^m$ form a mutually orthogonal basis for a timelike plane, where all geometric quantities associated with ${\mathbb G}$ will be denoted with a bar. This may only be locally defined, but can be patched over 
compact sets.

In particular, given a manifold with metric of signature $(n,m)$, we can choose local coordinates $(x^i,y^\alpha)$ such that $\frac{\partial }{\partial x^i}$ are spacelike and $\frac{\partial }{\partial y^\alpha}$ are timelike. Then the local functions $t_\alpha:p\mapsto y^\alpha(p)$ are multi-time functions.
\vspace{0.1in}

\begin{Def}
The manifold $({\mathbb M},{\mathbb G})$ is said to satisfy the {\it timelike curvature condition} if, for any spacelike $n$-plane 
$P$ at a point in ${\mathbb M}$, the Riemann curvature tensor satisfies
\begin{equation}\label{e:tcc}
\sum_{i=1}^n {\mathbb G}(\overline{R}(X,\tau_i)X,\tau_i)\;\geq k \;{\mathbb G}(X,X),
\end{equation}
for some positive constant $k$, where $\{\tau_i\}_{i=1}^n$ form an orthonormal basis for $P$ and $X$ is any timelike vector orthogonal to $P$. Here we use the following convention for the Riemann curvature tensor
\[
\overline{R}(X,Y)Z=-\overline{\nabla}_X\overline{\nabla}_YZ+\overline{\nabla}_Y\overline{\nabla}_XZ+\overline{\nabla}_{[X,Y]}Z,
\]
for vector fields $X,Y,Z$.
\end{Def}
\vspace{0.1in}

\begin{Note}
Definition 1 generalizes the codimension one {\it timelike convergence condition} of General Relativity, employed for example in \cite{EaH}:
\[
\overline{R}ic(X,X)\geq 0.
\]
\end{Note}

\vspace{0.1in}

Fix an orthonormal frame on (${\mathbb M},{\mathbb G}$):
\[
\{e_i,T_\alpha\}_{i,\alpha=1}^{n,m} \qquad \mbox{ s.t.}\qquad {\mathbb G}(e_i,e_j)=\delta_{ij}
\qquad {\mathbb G}(T_\alpha,T_\beta)=-\delta_{\alpha\beta} \qquad {\mathbb G}(e_i,T_\alpha)=0,
\]
with
\[
T_\alpha=-\psi_\alpha \overline{\nabla}t_\alpha
\qquad\qquad \psi_\alpha^{-2}=-{\mathbb G}(\overline{\nabla}t_\alpha,\overline{\nabla}t_\alpha).
\]

\begin{Def}\label{d:norm}
Given a contravariant tensor $B$ on ${\mathbb M}$ we define its norm by
\[
\|B\|^2=\sum_{i_1,...,i_l=1}^n [B(e_{i_1},e_{i_2},...,e_{i_l})]^2+\sum_{\beta_1,...,\beta_l=1}^m [B(T_{\beta_1},T_{\beta_2},...,T_{\beta_l})]^2.
\]
Similarly, for a covariant tensor $B$ we dualize with the metric ${\mathbb G}$ and define its norm as above. Note that this is not the usual Hilbert-Schmidt inner product on multi-linear functions, as it depends on the choice of an orthonormal frame.

Higher derivative norms are also defined:
\[
\|B\|^2_k=\sum_{j=0}^k\|\overline{\nabla}^jB\|^2.
\]

For a mixed tensor, we occasionally use the induced metric on the spacelike components to define a norm on the timelike components. That is, if $B_{\alpha\beta ijk}$ is a tensor of the indicated type, then we define
\[
|B_{\alpha\beta}|^2=\sum_{i,j,k=1}^n \|B_{\alpha\beta}(e_i,e_j,e_k)\|^2.
\]
\end{Def}

Let $f:\Sigma\rightarrow {\mathbb M}$ be a spacelike immersion of an $n$-dimensional manifold $\Sigma$, and let $g$ be the metric induced on $\Sigma$ by ${\mathbb G}$. 

\begin{Def}\label{d:adptframe}
A second orthonormal frame $\{\tau_{i},\nu_\alpha\}$ for (${\mathbb M},{\mathbb G}$) along $\Sigma$ is {\it adapted} to the submanifold if:
\[
\{\tau_i,\nu_\alpha\}_{i,\alpha=1}^{n,m} \qquad \mbox{ s.t.}\qquad {\mathbb G}(\tau_i,\tau_j)=\delta_{ij}
\qquad {\mathbb G}(\nu_\alpha,\nu_\beta)=-\delta_{\alpha\beta} \qquad {\mathbb G}(\tau_i,\nu_\alpha)=0,
\]
where $\{\tau_i\}_{i=1}^{n}$ form an orthonormal basis for ($\Sigma,g$), and $\{\nu_\alpha\}_{\alpha=1}^m$ span the normal space.

\end{Def}

The {\it second fundamental form} of the immersion is
\[
A_{ij\alpha}={\mathbb G}(\overline{\nabla}_{\tau_i}\nu_\alpha,\tau_j)=-{\mathbb G}(\overline{\nabla}_{\tau_i}\tau_j,\nu_\alpha),
\]
while the {\it mean curvature vector} is
\[
H_\alpha=g^{ij}A_{ij\alpha}.
\]
We have the following two equations for the splitting of the
connection
\begin{equation}\label{e:connsplit1}
\overline{\nabla}_{\tau_i}\tau_j=\nabla_{\tau_i}\tau_j-A_{ij}^\alpha\nu_\alpha
\end{equation}
\begin{equation}\label{e:connsplit2}
\overline{\nabla}_{\tau_i}\nu_\alpha=A^j_{i\alpha}\tau_j+C_{i\alpha}^\beta\nu_\beta,
\end{equation}
where $\nabla$ is the induced connection and $C_{i\alpha}^\beta$ are the components of the normal
connection
\begin{equation}\label{e:normconn}
\overline{\nabla}^\bot_{\tau_i}\nu_\alpha=C_{i\alpha}^\beta\nu_\beta.
\end{equation}

\section{Multi-angles}\label{s:anglesgen}

We now consider how to use orthonormal frames to define a matrix of angles between two spacelike $n$-planes in an $n+m$-manifold.

For frames $\{e_{i},T_\alpha\}$ and $\{\tau_{i},\nu_\alpha\}$ as above, introduce the notation
\[
X_{ij}={\mathbb G}(\tau_i,e_j)
\quad
W_{i\beta}={\mathbb G}(\tau_i,T_\beta)
\quad
U_{\alpha j}=-{\mathbb G}(\nu_\alpha,e_j)
\quad
V_{\alpha\beta}=-{\mathbb G}(\nu_\alpha,T_\beta).
\]
Thus
\[
e_i=X_{ji}\tau_j+U_{\alpha i}\nu_\alpha
\quad
T_\beta=W_{i\beta}\tau_i+V_{\alpha\beta}\nu_\alpha,
\]
and the $(n+m)\times (n+m)$ dimensional matrix
\[
M=\left(
     \begin{array}{cc}
      X & W \\
      -U & -V
     \end{array}\right),
\]
is an element of the indefinite orthogonal group $O(n,m)$.

\begin{Prop}
With notation as above, the $O(n,m)$ condition on M reads
\begin{equation}\label{e:onm1a}
X^TX=I_n+U^TU
\qquad
V^TV=I_m+W^TW
\qquad
U^TV=X^TW.
\end{equation}
\end{Prop}
\begin{proof}
This follows from the requirement that
\[
M^T\left(
     \begin{array}{cc}
      I_n & 0 \\
       0 & -I_m
     \end{array}\right)M=\left(
     \begin{array}{cc}
      I_n & 0 \\
       0 & -I_m
     \end{array}\right).
\]
\end{proof}

The vectors $\{\tau_i\}_1^n$ span the tangent space of $\Sigma$, while $\{\nu_\alpha\}_1^m$ span the normal bundle. We are free to rotate these frames within these two spaces, and this corresponds to left action of $O(n)$ and $O(m)$ on $O(n,m)$.

Similarly, we consider rotations of $\{e_i\}_1^n$ that preserve the $n$-dimensional vector space that they span, along with rotations of $\{T_\beta\}_1^m$ that preserves the $m$-dimensional space they span. These correspond to right actions of $O(n)$ and $O(m)$ within $O(n,m)$. Note that the positive definite norm in Definition \ref{d:norm} is preserved by these rotations.

\begin{Prop}\label{p:gauge}
By rotations of the frames $\{e_i,T_\alpha\}$ and $\{\tau_j,\nu_\beta\}$, which preserve the tangent and normal bundles of $\Sigma$ as well as the tensor norm of Definition \ref{d:norm}, we can simplify the matrix $M\in O(n,m)$ for $n\geq m$ to
\[
M=\left(
     \begin{array}{ccc}
      I_{n-m} & 0 & 0\\
       0 &  D_1 & \pm D_4A^T \\
       0 & AD_3 & D_2
     \end{array}\right),
\]
where $A\in O(m)$, $D_1$, $D_2$, $D_3$ and $D_4$ are diagonal matrices satisfying
\[
D_1^2=I_m+D_3^2
\qquad
D_2^2=I_m+D_4^2
\qquad
|D_1|^2=|D_2|^2,
\]
and $\pm$ of a diagonal matrix means a free choice of sign on the entries of the matrix. 

The case $n<m$ has a similar decomposition with $n$ and $m$ interchanged in the above formulae.
\end{Prop}
\begin{proof}
Consider first the matrix $X_{ij}=<\tau_i,e_j>$. The matrix $X^TX$ is symmetric and non-negative definite and so it has a well-defined square root, namely a symmetric $n\times n$ matrix which we denote by $\sqrt{X^TX}$. By the first equation of (\ref{e:onm1a}), $X$ is invertible since det$(X)\ge 1$ and so we can define the $n\times n$ matrix $A=\sqrt{X^TX}X^{-1}$. Then
\[
A^TI_nA=(X^{-1})^T\sqrt{X^TX}\sqrt{X^TX}X^{-1}=(X^{-1})^TX^TXX^{-1}=I_n,
\]
so that $A\in O(n)$. Define a new frame by $\{A_{ij}\tau_j,\nu_\alpha\}$ and then
\[
\tilde{X}_{ij}=A_{ik}<\tau_k,e_j>=\sqrt{X^TX}X^{-1}X=\sqrt{X^TX},
\]
which is symmetric. Now we can act on both the left and right of $\tilde{X}$ by $O(n)$ to diagonalize it.

A similar argument yields a diagonalization of $V_{\alpha\beta}$.

After diagonalization of $X$, the first of equations (\ref{e:onm1a}) implies that the matrix $U^TU$ is diagonal. Thus the $n$ $m$-dimensional vectors $\{U_{\alpha i}\nu_\alpha\}_{i=1}^n$ are mutually orthogonal and, since $n\geq m$, we conclude that $n-m$ of these vectors must be zero.

After a reordering of the basis elements, the matrix $M$ then decomposes into
\[
M=\left(
     \begin{array}{ccc}
      I_{n-m} & 0 & W_2\\
       0 & X_1 & W_1 \\
       0 & U_1 & V
     \end{array}\right).
\]
The last of equations (\ref{e:onm1a}) now implies that $W_2=0$ and we reduce the problem to the square case:
\[
X_1^TX_1=I_m+U_1^TU_1
\qquad
V^TV=I_m+W_1^TW_1
\qquad
U_1^TV=X_1^TW_1.
\]

In fact, to indicate that $X_1$ and $V$ are diagonal, let us write $X_1=D_1$ and $V=D_2$. Thus
\begin{equation}\label{e:onm1}
D_1^2=I_m+U_1^TU_1,
\end{equation}
\begin{equation}\label{e:onm2}
D_2^2=I_m+W_1^TW_1,
\end{equation}
\begin{equation}\label{e:onm3}
U_1^TD_2=D_1W_1.
\end{equation}
Equations (\ref{e:onm1}) and (\ref{e:onm2}) imply that there exists diagonal matrices $D_3$ and $D_4$ (with entries defined up to a sign) such that
\[
U_1=AD_3 \qquad W_1=BD_4 \qquad {\mbox{ for some }}A,B\in O(m).
\]
Thus equations (\ref{e:onm1}), (\ref{e:onm2}) and (\ref{e:onm3}) now read
\begin{equation}\label{e:onm4}
D_1^2=I_m+D_3^2,
\end{equation}
\begin{equation}\label{e:onm5}
D_2^2=I_m+D_4^2,
\end{equation}
\begin{equation}\label{e:onm6}
D_3A^TD_2=D_1BD_4.
\end{equation}
Taking the transpose of this last equation, multiplying across by the inverses of $D_1$ and $D_2$ (which exist by equations (\ref{e:onm4}) and (\ref{e:onm5})), and multiplying back on the right hand-side we find that
\begin{equation}\label{e:onm7}
AD_1^{-2}D_3^2A^T=D_2^{-2}D_4^2.
\end{equation}
Similarly
\[
BD_2^{-2}D_4^2B^T=D_1^{-2}D_3^2,
\]
and so $A=\pm B^T$.

Moreover, if $A\in O(m)$ conjugates a diagonal matrix to a diagonal matrix, then $A$ must permute the diagonal elements. 
Denote the  diagonal elements of $D_1$, $D_2$, $D_3$ and $D_4$ by $\lambda_i$, $\mu_i$, $a_i$ and $b_i$, respectively, where $i=n-m+1,...,n$.
Then equations (\ref{e:onm4}), (\ref{e:onm5}) and (\ref{e:onm7}) read
\[
\lambda_i^2=1+a_i^2
\qquad
\mu_i^2=1+b_i^2
\qquad
\mu_i^2a_i^2=\lambda_{p(i)}^2b_{p(i)}^2,
\]
where $p$ is the permutation of $(n-m+1,...,n)$ determined by $A$. Combining these three equations we get
\[
a_i^2+a_i^2b_i^2=b_{p(i)}^2+a_{p(i)}^2b_{p(i)}^2,
\]
which when summed yields
\[
\sum_ia_i^2=\sum_ib_i^2
\qquad\quad
{\mbox{and}}
\qquad\quad
\sum_i\lambda_i^2=\sum_i\mu_i^2.
\]
Thus $|D_1|^2=|D_2|^2$, where for any diagonal matrix $D$, $|D|^2=tr(D^2)$.

\end{proof}

\vspace{0.1in}

\begin{Def}
The function $v$ is defined to be
\[
v^2=V^{\alpha\beta}V_{\alpha\beta},
\]
where $V^{\alpha\beta}=-{\mathbb G}(\nu^\alpha,T^\beta)$, with respect to the dual coframes $\{e^{i},T^\alpha\}$ and $\{\tau^{i},\nu^\alpha\}$. This is a generalization of the {\it tilt function} in the case of codimension one appearing in  \cite{Bart}.
\end{Def}

\vspace{0.1in}

We now use the normal form to construct estimates for the norm of the adapted frames in terms of $v$:

\vspace{0.1in}

\begin{Prop}\label{p:frameest}
For an adapted frame $\{\tau_i,\nu_\alpha\}$ we have
\[
\|\tau_i\|^2\leq n(n+2) v^2 \qquad \|\nu_\alpha\|^2\leq 2mv^2,
\]
for all $i=1,2,...,n$ and $\alpha=1,2,...,m$.
\end{Prop}
\begin{proof}
Any adapted frame $\{\tau_i,\nu_\alpha\}$ can be related by rotations  $A\in O(n)$ and $B\in O(m)$ to an 
adapted frame $\{\mathring{\tau}_i,\mathring{\nu}_\alpha\}$ for which, with respect to an orthonormal background  basis $\{\mathring{e}_i,\mathring{T}_\alpha\}$, the matrix $M$ has the form given in Proposition \ref{p:gauge}. 

That is,
\[
\tau_i=A_i^j\mathring{\tau}_j
\qquad
\nu_\alpha=B_\alpha^\beta\mathring{\nu}_\beta.
\]
Then
\begin{align}
\|\tau_i\|^2&=\sum_j({\mathbb G}(\tau_i,\mathring{e}_j))^2+\sum_\alpha({\mathbb G}(\tau_i,\mathring{T}_\alpha))^2\nonumber\\
&=\sum_j\left[\sum_kA_i^k{\mathbb G}(\mathring{\tau}_k,\mathring{e}_j)\right]^2
    +\sum_\alpha\left[\sum_kA_i^k{\mathbb G}(\mathring{\tau}_k,\mathring{T}_\alpha)\right]^2 \nonumber\\
&\leq\sum_j\left[\sum_k|A_i^k|\;|{\mathbb G}(\mathring{\tau}_k,\mathring{e}_j)|\right]^2
    +\sum_\alpha\left[\sum_k|A_i^k|\;|{\mathbb G}(\mathring{\tau}_k,\mathring{T}_\alpha)|\right]^2\nonumber\\
&\leq\sum_j\left[\sum_k\;|{\mathbb G}(\mathring{\tau}_k,\mathring{e}_j)|\right]^2
    +\sum_\alpha\left[\sum_k\;|{\mathbb G}(\mathring{\tau}_k,\mathring{T}_\alpha)|\right]^2\nonumber\\
&=\sum_j\left[\sum_k\;|X_{kj}|\right]^2+\sum_\alpha\left[\sum_k\;|W_{k\alpha}|\right]^2\nonumber\\
&\leq n\sum_{j,k}\;|X_{jk}|^2+n\sum_{\alpha,k}|W_{k\alpha}|^2\nonumber\\
&= n(n-m+|D_1|^2)+n|D_4|^2\nonumber\\
&= n(n-m+v^2)+n(v^2-m)\nonumber\\
&\leq n(n+2)v^2\nonumber.
\end{align}

Similarly for $\nu_\alpha$:
\begin{align}
\|\nu_\beta\|^2&=\sum_j({\mathbb G}(\nu_\beta,\mathring{e}_j))^2+\sum_\alpha({\mathbb G}(\nu_\beta,\mathring{T}_\alpha))^2\nonumber\\
&=\sum_j\left[\sum_\gamma B_\beta^\gamma{\mathbb G}(\mathring{\nu}_\gamma,\mathring{e}_j)\right]^2
    +\sum_\alpha\left[\sum_\gamma B_\beta^\gamma{\mathbb G}(\mathring{\nu}_\gamma,\mathring{T}_\alpha)\right]^2 \nonumber\\
&\leq m\sum_{\gamma, j}\;|U_{\gamma j}|^2+m\sum_{\alpha,\gamma}|V_{\alpha\gamma}|^2\nonumber\\
&= m(|D_3|^2+v^2)\nonumber\\
&= m(|D_1|^2-m+v^2)\nonumber\\
&\leq 2mv^2\nonumber.
\end{align}

\end{proof}
\vspace{0.1in}

\section{The height functions}
Let $u_\alpha:\Sigma\rightarrow {\mathbb R}$ be the {\it height function} $u_\alpha=t_\alpha\circ f$. We now prove

\vspace{0.1in}
\begin{Prop}\label{p:slice}
For all $\alpha=1,...,m$ we have
\[
\nabla u_\alpha=\overline{\nabla}t_\alpha+\psi^{-1}_\alpha\sum_\beta V_{\beta\alpha}\nu_\beta,
\]
\[
\nabla u_\alpha\cdot\nabla u_\beta=\psi^{-1}_\alpha\psi^{-1}_\beta\left(\sum_\gamma V_{\gamma\alpha}V_{\gamma\beta}-\delta_{\alpha\beta}\right).
\]
\end{Prop}
\begin{proof}
From the definition of $u_\alpha$ and $T_\alpha$ we have 
\[
\nabla u_\alpha=\overline{\nabla}t_\alpha+\psi^{-1}_\alpha\sum_\beta V_{\beta\alpha}\nu_\beta
=\psi^{-1}_\alpha\left(\sum_\beta V_{\beta\alpha}\nu_\beta-T_\alpha\right),
\]
and so
\begin{align}
\nabla u_\alpha\cdot\nabla u_\beta=&\psi^{-1}_\alpha\psi^{-1}_\beta\;{\mathbb G}\left(\sum_\gamma V_{\gamma\alpha}\nu_\gamma-T_\alpha,
                      \sum_\delta V_{\delta\beta}\nu_\delta-T_\beta\right)\nonumber\\
&=\psi^{-1}_\alpha\psi^{-1}_\beta\left(\sum_\gamma V_{\gamma\alpha}V_{\gamma\beta}-\delta_{\alpha\beta}\right)\nonumber.
\end{align}
as claimed.
\end{proof}

\begin{Prop}\label{p:laps}
\[
\triangle u_\gamma=-\psi_\gamma^{-1}V_{\alpha\gamma}H^\alpha+g^{ij}\overline{\nabla}_i\overline{\nabla}_jt_\gamma.
\]
\begin{align}
\triangle V_{\alpha\beta}=& V_{\gamma\beta}(A_{ij\gamma}A^{ij}_\alpha+<\overline{\mbox {R}}(\tau_i,\nu_\gamma)\tau_i,\nu_\alpha>)
      -\overline{\nabla}^\bot _{T_\beta}H_\alpha-A^{ij}_\alpha T_\beta(g_{ij})\nonumber\\
&\qquad +{\textstyle{\frac{1}{2}}}(\overline{\nabla}{\mathcal{L}}_{T_\beta}{\mathbb G})(\nu_\alpha,\tau_i,\tau_i)
    -(\overline{\nabla}{\mathcal{L}}_{T_\beta}{\mathbb G})(\tau_i,\nu_\alpha,\tau_i)-(\overline{\nabla}T_\beta)(H,\nu_\alpha)\nonumber\\
&\qquad -2C_{i\alpha}^{\;\;\;\;\gamma}<\nu_\gamma,\overline{\nabla}_{T_\beta}\tau_i>+(\nabla_{\tau_i} C_{i\alpha}^{\;\;\;\;\gamma}
    +C_{i\alpha}^{\;\;\;\;\delta}C_{i\delta}^{\;\;\;\;\gamma})V_{\gamma\beta}\nonumber,
\end{align}
where $\triangle$ is the Laplacian of the induced metric given by $\triangle =g^{ij}\nabla_{\tau_i}\nabla_{\tau_j}$ and $\overline{\nabla}^\bot $ is the normal connection, as defined in equation (\ref{e:normconn}).
\end{Prop}
\begin{proof}
The first statement follows from a straightforward generalization of the codimension one case \cite{EaH}.

For the second statement we follow Proposition 2.1 of Bartnik \cite{Bart}, fix a point $p\in\Sigma$ and choose an orthonormal frame $\{\tau_i\}$ on $\Sigma$ such that $(\nabla_{\tau_i}\tau_j)(p)=0$. Extend this frame to a neighbourhood of $\Sigma$ satisfying ${\mathcal L}_{T_\beta}\tau_i=[T_\beta,\tau_i]=0$ for a fixed $\beta$. Then (here, there is no summation over $\alpha$ or $\beta$, but there is over other repeated indices)
\begin{align}
-\triangle V_{\alpha\beta}&= \triangle <\nu_\alpha,T_\beta>=\overline{\nabla}_{\tau_i}\overline{\nabla}_{\tau_i}<\nu_\alpha,T_\beta>\nonumber\\
&= \overline{\nabla}_{\tau_i}(<\overline{\nabla}_{\tau_i}\nu_\alpha,T_\beta>+<\nu_\alpha,\overline{\nabla}_{\tau_i}T_\beta>)\nonumber\\
&= \overline{\nabla}_{\tau_i}(A_{i\alpha}^j<\tau_j,T_\beta>+C_{i\alpha}^\gamma<\nu_\gamma,T_\beta>+<\nu_\alpha,\overline{\nabla}_{\tau_i}T_\beta>).\nonumber
\end{align}
Computing each of the three terms in turn,
\[
\overline{\nabla}_{\tau_i}(A_{i\alpha}^j<\tau_j,T_\beta>)=\overline{\nabla}_{\tau_i}A_{i\alpha}^j<\tau_j,T_\beta>+A_{i\alpha}^j(<\overline{\nabla}_{\tau_i}\tau_j,T_\beta>+<\tau_j,\overline{\nabla}_{\tau_i}T_\beta>),
\]
and by the Gauss equation we have that
\[
\overline{\nabla}_{\tau_i}A_{i\alpha}^j=\overline{\nabla}_{\tau_j}H_\alpha-<\overline{R}(\tau_i,\tau_j)\nu_\alpha,\tau_i>-A_{ij}^\gamma C_{i\gamma}^\alpha +H^\gamma C_{j\gamma}^\alpha.
\]
For the second term we have
\[
\overline{\nabla}_{\tau_i}(C_{i\alpha}^\gamma<\nu_\gamma,T_\beta>)=<\nu_\gamma,T_\beta>\overline{\nabla}_{\tau_i}C_{i\alpha}^\gamma+C_{i\alpha}^\gamma(<\overline{\nabla}_{\tau_i}\nu_\gamma,T_\beta>+<\nu_\gamma,\overline{\nabla}_{\tau_i}T_\beta>),
\]
while for the third term
\[
\overline{\nabla}_{\tau_i}(<\nu_\alpha,\overline{\nabla}_{\tau_i}T_\beta>)=-<\overline{R}(\tau_i,T_\beta)\tau_i,\nu_\alpha>+<\nu_\alpha,\overline{\nabla}_{T_\beta}\overline{\nabla}_{\tau_i}\tau_i> +<\overline{\nabla}_{\tau_i}\nu_\alpha,\overline{\nabla}_{T_\beta}\tau_i>,
\]
where we have used the fact that $[T_\beta,\tau_i]=\overline{\nabla}_{T_\beta}\tau_i-\overline{\nabla}_{\tau_i}T_\beta=0$, and commuted the second derivatives, which brings in the curvature term. Assembling the three terms yields
\begin{align}
-\triangle V_{\alpha\beta}&= -<\overline{R}(\tau_i,T_\beta)\tau_i,\nu_\alpha>+<\nu_\alpha,\overline{\nabla}_{T_\beta}\overline{\nabla}_{\tau_i}\tau_i>
  +<\overline{\nabla}_{\tau_i}\nu_\alpha,\overline{\nabla}_{T_\beta}\tau_i>\nonumber\\
&\qquad\quad +(\overline{\nabla}_{\tau_j}H_\alpha-<\overline{R}(\tau_i,\tau_j)\nu_\alpha,\tau_i>-A_{ij}^\gamma C_{i\gamma}^\alpha
  +H^\gamma C_{j\gamma}^\alpha)<\tau_j,T_\beta>\nonumber\\
&\qquad\quad +C_{i\alpha}^\gamma(<\overline{\nabla}_{\tau_i}\nu_\gamma,T_\beta>+<\nu_\gamma,\overline{\nabla}_{\tau_i}T_\beta>)
  +<\nu_\gamma,T_\beta>\overline{\nabla}_{\tau_i}C_{i\alpha}^\gamma \nonumber\\
&\qquad\quad +A_{i\alpha}^j(<\overline{\nabla}_{\tau_i}\tau_j,T_\beta>+<\tau_j,\overline{\nabla}_{\tau_i}T_\beta>)\nonumber\\
&= <\overline{R}(\tau_i,\nu_\gamma)\tau_i,\nu_\alpha><\nu_\gamma,T_\beta>+<\nu_\alpha,\overline{\nabla}_{T_\beta}\overline{\nabla}_{\tau_i}\tau_i>
  +2A_{i\alpha}^j<\tau_j,\overline{\nabla}_{T_\beta}\tau_i>\nonumber\\
&\qquad\quad +2C_{i\alpha}^\gamma<\nu_\gamma,\overline{\nabla}_{\tau_i}T_\beta>+<\tau_i,T_\beta>\overline{\nabla}_{\tau_i}H_\alpha
  +H^\gamma C_{i\gamma\alpha}<\tau_i,T_\beta>\nonumber\\
&\qquad\quad +A_{\alpha}^{ij} A_{ij}^\gamma<\nu_\gamma,T_\beta>+C_{i\alpha}^{\gamma} C_{i\gamma}^\delta<\nu_\delta,T_\beta>
+<\nu_\gamma,T_\beta>\overline{\nabla}_{\tau_i}C_{i\alpha}^\gamma \nonumber\\
&=-V_{\gamma\beta}(A_{ij\gamma}A^{ij}_\alpha+<\overline{\mbox {R}}(\tau_i,\nu_\gamma)\tau_i,\nu_\alpha>)
      +\overline{\nabla}^\bot _{T_\beta}H_\alpha+<\nu_\alpha,\overline{\nabla}_{T_\beta}\overline{\nabla}_{\tau_i}\tau_i>\nonumber\\
&\qquad +A_\alpha^{ij}T_\beta<\tau_i,\tau_j>+2C_{i\alpha}^{\;\;\;\;\gamma}<\nu_\gamma,\overline{\nabla}_{T_\beta}\tau_i>+(\nabla_{\tau_i} C_{i\alpha}^{\;\;\;\;\gamma}-C_{i\alpha}^{\;\;\;\;\delta}C_{i\delta}^{\;\;\;\;\gamma})V_{\gamma\beta}\nonumber.
\end{align}
The second equality uses the fact that
\begin{align}
-<\overline{R}(\tau_i,T_\beta)\tau_i,\nu_\alpha>&=-<\overline{R}(\tau_i,\tau_j)\tau_i,\nu_\alpha><\tau_j,T_\beta>+<\overline{R}(\tau_i,\nu_\gamma)\tau_i,\nu_\alpha><\nu_\gamma,T_\beta>\nonumber\\
&=<\overline{R}(\tau_i,\tau_j)\nu_\alpha,\tau_i><\tau_j,T_\beta>+<\overline{R}(\tau_i,\nu_\gamma)\tau_i,\nu_\alpha><\nu_\gamma,T_\beta>,\nonumber
\end{align}
and, as per equation (\ref{e:connsplit2}), the substitution
\[
\overline{\nabla}_{\tau_i}\nu_\alpha=A^j_{i\alpha}\tau_j+C_{i\alpha}^\gamma\nu_\gamma,
\]
while by equation (\ref{e:connsplit1}) and the assumption $(\nabla_{\tau_i}\tau_j)(p)=0$, we utilize
\[
\overline{\nabla}_{\tau_i}\tau_j=\nabla_{\tau_i}\tau_j-A_{ij}^\gamma\nu_\gamma=-A_{ij}^\gamma\nu_\gamma.
\]
The final equality comes from gathering terms and using the definition of $V_{\alpha\beta}$. We now use the following:

\begin{Lem}\label{l:lie}
\begin{align}
T_\beta<\tau_i,\overline{\nabla}_{\tau_i}\nu_\alpha>&=-<\overline{\nabla}_{\tau_i}\tau_i,\overline{\nabla}_{T_\beta}\nu_\alpha>+
   {\textstyle{\frac{1}{2}}}(\overline{\nabla}{\mathcal{L}}_{T_\beta}{\mathbb G})(\nu_\alpha,\tau_i,\tau_i)\nonumber\\
  &\qquad  -(\overline{\nabla}{\mathcal{L}}_{T_\beta}{\mathbb G})(\tau_i,\nu_\alpha,\tau_i)-<\overline{\nabla}_HT_\beta,\nu_\alpha>\nonumber.
\end{align}
\end{Lem}
\begin{proof}
The proof of this follows the codimension one case (Proposition 2.1 of \cite{Bart}).
\end{proof}

To complete the proof of the Proposition, note that
\begin{align}
<\nu_\alpha,\overline{\nabla}_{T_\beta}\overline{\nabla}_{\tau_i}\tau_i>&=T_\beta<\nu_\alpha,\overline{\nabla}_{\tau_i}\tau_i>
     -<\overline{\nabla}_{T_\beta}\nu_\alpha,\overline{\nabla}_{\tau_i}\tau_i>\nonumber\\
&=-T_\beta<\overline{\nabla}_{\tau_i}\nu_\alpha,\tau_i>-<\overline{\nabla}_{T_\beta}\nu_\alpha,\overline{\nabla}_{\tau_i}\tau_i>\nonumber\\
&=-{\textstyle{\frac{1}{2}}}(\overline{\nabla}{\mathcal{L}}_{T_\beta}{\mathbb G})(\nu_\alpha,\tau_i,\tau_i)
    +(\overline{\nabla}{\mathcal{L}}_{T_\beta}{\mathbb G})(\tau_i,\nu_\alpha,\tau_i)+<\overline{\nabla}_HT_\beta,\nu_\alpha>\nonumber,
\end{align}
where in the last equality we have used Lemma \ref{l:lie}. Substituting this in the second equation of Proposition \ref{p:laps} then yields the result.

\end{proof}

\vspace{0.1in}

\section{The initial value problem}\label{s4}

Let $f_s:\Sigma\rightarrow{{\mathbb M}}$  for $s\in[0,s_0)$ be a family of compact $n$-dimensional spacelike immersed submanifold in an $n+m$-dimensional manifold ${\mathbb M}$ with a metric ${\mathbb G}$ of signature $(n,m)$. In addition, we assume that $n\geq m$. The case $n< m$ follows by similar arguments.

Then $f_s$ moves by parameterized mean curvature flow if it satisfies the following initial value problem:

\vspace{0.1in}
{\it
Let $f_s:\Sigma\rightarrow{{\mathbb M}}$ be a family of spacelike immersed submanifolds satisfying
\[
\frac{d f}{ds}=H,
\]
with initial conditions 
\[
f_0(\Sigma)=\Sigma_0,
\] 
where $H$ is the mean curvature vector associated with the immersion $f_s$ in $({\mathbb M},{\mathbb G})$, and $\Sigma_0$ is some given initial compact $n$-dimensional spacelike immersed submanifold.
}
\vspace{0.1in}

The evolution of the functions $u_\gamma$ and $v$ is then given by:

\begin{Prop}\label{p:flowuv}
\begin{equation}\label{e:uflow}
\left(\frac{d}{ds}-\triangle\right)u_\gamma=-g^{ij}\overline{\nabla}_i\overline{\nabla}_jt_\gamma,
\end{equation}
\begin{align}
v\left(\frac{d}{ds}-\triangle\right) v\leq& -V^{\alpha\beta}V_{\gamma\beta}
      (A_{ij\gamma}A^{ij}_\alpha+<\overline{\mbox {R}}(\tau_i,\nu_\gamma)\tau_i,\nu_\alpha>)+A^{ij}_\alpha {\mathcal L}_{T_\beta}g_{ij}V^{\alpha\beta}\nonumber\\
&\qquad -{\textstyle{\frac{1}{2}}}(\overline{\nabla}{\mathcal{L}}_{T_\beta}{\mathbb G})(\nu_\alpha,\tau_i,\tau_i)V^{\alpha\beta}
    +(\overline{\nabla}{\mathcal{L}}_{T_\beta}{\mathbb G})(\tau_i,\nu_\alpha,\tau_i)V^{\alpha\beta}\nonumber\\
&\qquad +2C_{i\alpha}^{\;\;\;\;\gamma}<\nu_\gamma,\overline{\nabla}_{T_\beta}\tau_i>V^{\alpha\beta}
    -C_{i\alpha}^{\;\;\;\;\delta}C_{i\delta}^{\;\;\;\;\gamma}V_{\gamma\beta}V^{\alpha\beta}\label{e:muflow}.
\end{align}
\end{Prop}
\begin{proof}
Generalizing Proposition 3.1 of \cite{EaH}, note the time derivatives are
\[
\frac{du_\gamma}{ds}=-\psi_\gamma^{-1}V_{\alpha\gamma}H^\alpha,
\]
\[
\frac{dV_{\alpha\beta}}{ds}=-\overline{\nabla}_{T_\beta}H_\alpha-H^\gamma<\overline{\nabla}_{\nu_\gamma}T_\beta,\nu_\alpha>.
\]
This last equation follows from
\begin{align}
\frac{dV_{\alpha\beta}}{ds}=&-\frac{d{\mathbb G}(\nu_\alpha,T_\beta)}{ds}\nonumber\\
=&-\frac{d{\mathbb G}}{ds}(\nu_\alpha,T_\beta)-{\mathbb G}(\frac{d\nu_\alpha}{ds},T_\beta)-{\mathbb G}(\nu_\alpha,\frac{dT_\beta}{ds})\nonumber\\
=&-\overline{\nabla}_{T_\beta}H_\alpha-{\mathbb G}(\nu_\alpha,\overline{\nabla}_{H}T_\beta).\nonumber
\end{align}
The flow of $u_\gamma$ then follows immediately from Proposition \ref{p:laps}.

The evolution of the tilt function $v^2 = V^{\alpha\beta}V_{\alpha\beta}$  note that, since
\[
v=\sqrt{\sum_{\alpha,\beta} V_{\alpha\beta}V_{\alpha\beta}}
\]
we have
\[
\nabla v=\frac{1}{v}\sum_{\alpha,\beta} V_{\alpha\beta}\nabla V_{\alpha\beta} \qquad \frac{d}{ds}V_{\beta}=\frac{1}{v}\sum_{\alpha,\beta} V_{\alpha\beta}\frac{d}{ds} V_{\alpha\beta} 
\]
and
\[
\nabla \nabla v=\frac{1}{v}\sum_{\alpha,\beta}( V_{\alpha\beta}\nabla \nabla V_{\alpha\beta}+\nabla V_{\alpha\beta}\nabla V_{\alpha\beta})-\frac{1}{v^3}\left(\sum_{\alpha,\beta} V_{\alpha\beta}\nabla V_{\alpha\beta}\right)\left(\sum_{\gamma,\delta} V_{\gamma\delta}\nabla V_{\gamma\delta}\right).
\]
Taking the trace and rearranging 
\begin{align}
v\left(\frac{d}{ds}-\triangle\right)v=&\sum_{\alpha,\beta}V_{\alpha\beta}\left(\frac{d}{ds}-\triangle\right)V_{\alpha\beta}\nonumber\nonumber\\
  &+\frac{1}{v^2}\sum_{\alpha,\beta,\gamma,\delta}\left[(V_{\alpha\beta}\nabla V_{\alpha\beta})\cdot (V_{\gamma\delta}\nabla V_{\gamma\delta}) -
  (V_{\alpha\beta}V_{\alpha\beta}) (\nabla V_{\gamma\delta}\cdot \nabla V_{\gamma\delta})\right] \nonumber.
\end{align}
The expression in the square bracket is non-positive since for scalars $a_k$ and vectors $v_k$ in an inner product space we have 
\begin{align}
\sum_k a_k v_k \cdot \sum_l a_l v_l &= \sum_{k,l}a_k a_l v_k \cdot v_l
\le \frac{1}{4} \sum_{k,l}(a_k^2 + a_l^2)(|v_k|^2 + |v_l|^2) 
\nonumber\\&= \sum_k a_k^2 |v_k|^2 + \frac{1}{2} \sum_{k \neq l}a_k^2 |v_l|^2
\le \left(\sum_k a_k^2\right)\left(\sum_l |v_l|^2\right)\nonumber.
\end{align}
We conclude that
\[
v\left(\frac{d}{ds}-\triangle\right)v\leq V_{\alpha\beta}\left(\frac{d}{ds}-\triangle\right)V_{\alpha\beta}.
\]
Now contracting the second equation of Proposition \ref{p:laps} with $V_{\alpha\beta}$ yields the claim.
\end{proof}

\vspace{0.1in}

\begin{Prop}\label{p:gradest}
Assume that ${\mathbb M}$ satisfies the timelike curvature condition (\ref{e:tcc}). Let $\Sigma_s$ be a smooth solution of the initial value problem on the interval $0\leq s<s_0$ such that $\Sigma_s$ is contained in a compact subset of ${\mathbb M}$ for all $0\leq s<s_0$. Then the function $ v$ satisfies the a priori estimate
\[
 v(p,s)\leq(m+\sup_{\Sigma\times 0} v)\sup_{(q,s)\in \Sigma\times[0,s_0]}\exp[K(u(q,s)-u(p,s))],
\]
for some positive constant K$(n,m,\|t\|_3,|\psi|,\|\overline{R}\|,|H|,k)$, where $u=\sum_\alpha u_\alpha$.
\end{Prop}
\begin{proof}
The argument is an extension of Bartnik's estimate in the stationary case \cite{Bart} to the parabolic case in higher codimension.

Let K$>$0 be a constant to be determined later and set
\[
C_K=(m+\sup_{\Sigma\times 0} v)\sup_{\Sigma\times[0,s_0]}\exp(Ku).
\]
Consider the test function $h= v \exp(Ku)$. Suppose, for the sake of contradiction, that the function $h_K$ reaches $C_K$ for the first time at $(p_1,s_1)\in \Sigma\times(0,s_0]$. Then at this point $v\geq m+1$ and by the maximum principle
\[
\left(\frac{d}{ds}-\triangle\right) h_K{\geq}0
\qquad\qquad \nabla h_K{=}0.
\]
Here and throughout we evaluate all quantities at the point $(p_1,s_1)$. Moreover, for quantities that depend on normal indices, we choose an adapted orthonormal frame $\{\nu_\alpha\}_1^m$ which diagonalizes the matrix $V$: $V_{\alpha\beta}=V_\alpha\delta_{\alpha\beta}$ at $(p_1,s_1)$. 

Working out these two equations we have
\begin{equation}\label{e:atmax1}
\left(\frac{d}{ds}-\triangle\right)  v +K v
\left(\frac{d}{ds}-\triangle\right) u-2K\nabla u\cdot\nabla v -K^2
v |\nabla u|^2{\geq}0,
\end{equation}
\begin{equation}\label{e:atmax2}
\nabla  v +K v \nabla u{=}0.
\end{equation}
Substituting the second of these in the first we obtain
\begin{equation}\label{e:atmax3}
K v \left(\frac{d}{ds}-\triangle\right)
u{\geq}-\left(\frac{d}{ds}-\triangle\right) v -K^2
v |\nabla u|^2.
\end{equation}
From Proposition \ref{p:flowuv} and the estimates in
Proposition \ref{p:frameest}
\begin{equation}\label{e:boxuest}
\left(\frac{d}{ds}-\triangle\right)
u_\gamma=-g^{ij}\overline{\nabla}_i\overline{\nabla}_j t_\gamma\leq
\|\overline{\nabla}_i\overline{\nabla}_j
t_\gamma\|.\|\tau_i\|.\|\tau_j\|\leq C_1(n,m,\|t\|_2) v^2.
\end{equation}

We now simplify and estimate the terms that arise on the right hand side of equation (\ref{e:muflow}). At $p_1$ we may set $C_{i\alpha}^{\;\;\;\beta}=0$ and utilise the frame choice at $(p_1,s_1)$ mentioned above $V_{\alpha\beta}=V_\beta\delta_{\alpha\beta}$. Thus
\[
-\sum_{\alpha,\beta,\gamma}V^{\alpha\beta}V_{\gamma\beta}A_{ij\gamma}A^{ij}_\alpha=-\sum_{\alpha}V_\alpha^2|A_\alpha|^2,
\]
\[
\sum_{\alpha,\beta}A^{ij}_\alpha {\mathcal L}_{T_\beta}g_{ij}V^{\alpha\beta}\leq\sum_{\alpha}C_2(\|T\|_1)|A_\alpha|V_\alpha,
\]
\[
-{\textstyle{\frac{1}{2}}}\sum_{\alpha}(\overline{\nabla}{\mathcal{L}}_{T_\beta}{\mathbb G})(\nu_\alpha,\tau_i,\tau_i)V^{\alpha\beta}
    +\sum_{\alpha}(\overline{\nabla}{\mathcal{L}}_{T_\beta}{\mathbb G})(\tau_i,\nu_\alpha,\tau_i)V^{\alpha\beta}\leq C_3(n,m,\|T\|_2)v^4.
\]
Assembling this with the timelike curvature condition (\ref{e:tcc}) yields
\begin{align}
v\left(\frac{d}{ds}-\triangle\right)v{\leq}&
-\sum_{\alpha}V_\alpha^2|A_\alpha|^2+C_2|A_\alpha|V_\alpha+C_3v^4\nonumber\\
\leq&-(1-\epsilon)\sum_{\alpha}V_\alpha^2|A_\alpha|^2+C_4(\epsilon,n,m,\|T\|_2)v^4\label{e:eq11},
\end{align}
for any choice of $\epsilon>0$. 

Here the last inequality uses Young's inequality:
\begin{equation}\label{e:young}
ab\leq\frac{\epsilon a^2}{2}+\frac{b^2}{2\epsilon}.
\end{equation}

Now, for any symmetric matrix $M$ with eigenvalues $\lambda_i$, $i = 1,...,n $, we have the following inequalities
\[
\|M\|^2=\sum_{i=1}^n\lambda_i^2\geq\lambda_1^2+\frac{1}{n-1}\left(\sum_{i=2}^n\lambda_i\right)^2\geq\left(1+\frac{1}{n}\right)\lambda_1^2-\left(\sum_{i=1}^n \lambda_i\right)^2.
\]
The first inequality follows from the fact that
\[
\sum_{i=2}^n\lambda_i^2\geq \frac{1}{n-1}\left(\sum_{i=2}^n\lambda_i\right)^2
\]
while to prove the second inequality, let $a=\lambda_1$ and $b=\sum_{i=2}^n\lambda_i$, and compute
\begin{align}
(a+b)^2+\frac{1}{n-1}b^2-\frac{1}{n}a^2&=a^2+2ab+b^2+\frac{1}{n-1}b^2-\frac{1}{n}a^2\nonumber\\
&=2ab+\frac{n}{n-1}b^2-\frac{1-n}{n}a^2\nonumber\\
&=\frac{n}{n-1}\left(b^2+\frac{n-1}{n}a\right)^2+\frac{1-n}{n}a^2-\frac{1-n}{n}a^2\nonumber\\
&=\frac{n}{n-1}\left(b^2+\frac{n-1}{n}a\right)^2\geq 0,\nonumber
\end{align}
which implies that
\[
\left(\sum_{i=1}^n \lambda_i\right)^2+\frac{1}{n-1}\left(\sum_{i=2}^n\lambda_i\right)^2-\frac{1}{n}\lambda_1^2\geq0,
\]
as claimed.

Applying this to our case, this gives
\begin{equation}\label{e:ineq1}
\sum_{\alpha}V_\alpha^2|A_\alpha|^2\geq\sum_{\alpha}\left(1+\frac{1}{n}\right)\lambda_\alpha^2V_\alpha^2-H_\alpha^2V_\alpha^2,
\end{equation}
where $\lambda_\alpha$ is the eigenvalue of $A_{ij\alpha}$ with the maximum absolute value, so that in an eigenframe $A_{ij\alpha}\leq|\lambda_\alpha|\delta_{ij}$.

On the other hand we compute
\[
\nabla_{\tau_i}V_{\alpha\beta}=-A_{i\alpha}^j<\tau_j,T_\beta>-<\nu_\alpha,\overline{\nabla}_iT_\beta>,
\]
and so
\[
v\nabla_{\tau_i}v=V^{\alpha\beta}\nabla_{\tau_i}V_{\alpha\beta}=-A_{i\alpha}^jW_{j\beta}V^{\alpha\beta}-<\nu_\alpha,\overline{\nabla}_iT_\beta>V^{\alpha\beta}.
\]
The square norm is
\begin{align}
v^2|\nabla v|^2&=v^2\nabla_{\tau_i}v\nabla^iv\nonumber\\
&=\left(A_{i\alpha}^jW_{j\beta}+<\nu_\alpha,\overline{\nabla}_iT_\beta>\right)
\left(A_{\gamma}^{ik}W_{k\delta}+<\nu_\gamma,\overline{\nabla}^iT_\delta>\right)V^{\alpha\beta}V^{\gamma\delta}\nonumber
\\
&=A_{i\alpha}^jA_{\gamma}^{ik}W_{j\beta}W_{k\delta}V^{\alpha\beta}V^{\gamma\delta}+2A_{i\alpha}^jW_{j\beta}<\nu_\gamma,\overline{\nabla}^iT_\delta>V^{\alpha\beta}V^{\gamma\delta}\nonumber\\
&\qquad\qquad+<\nu_\alpha,\overline{\nabla}_iT_\beta><\nu_\gamma,\overline{\nabla}^iT_\delta>V^{\alpha\beta}V^{\gamma\delta}.\nonumber
\end{align}
Take these three summands separately, computing in a tangent eigenframe (so that $A_{ij\alpha}\leq|\lambda_\alpha|\delta_{ij}$). The first term is
\begin{align}
A_{i\alpha}^jA_{\gamma}^{ik}W_{j\beta}W_{k\delta}V^{\alpha\beta}V^{\gamma\delta}&\leq |\lambda_\alpha\lambda_\gamma| .|W^k_{\beta}W_{k\delta}
V^{\alpha\beta}V^{\gamma\delta}|\nonumber\\
&=|\lambda_\alpha\lambda_\gamma| .|\left(V^\rho_{\beta}V_{\rho\delta}-\delta_{\beta\gamma}\right)V^{\alpha\beta}V^{\gamma\delta}|\nonumber\\
&=\sum_{\alpha}\lambda^2_\alpha\left(V_{\alpha}^2-1\right)V^2_{\alpha}\nonumber\\
&\leq v^2\sum_{\alpha}\lambda^2_\alpha V^2_{\alpha},\nonumber
\end{align}
where we have used the relationship between the matrices $W$ and $V$ given in the middle of equations (\ref{e:onm1a}). Note that this equation implies $\|W_\beta\|^2=V_\beta^2-1\leq v^2-1\leq v^2$.

For the second term, again computing in an eigenframe for $V^{\alpha\beta}$,
\begin{align}
2A_{i\alpha}^jW_{j\beta}<\nu_\gamma,\overline{\nabla}^iT_\delta>V^{\alpha\beta}V^{\gamma\delta}
&\leq 2|\lambda_{\alpha}|.|W_{i\beta}<\nu_\gamma,\overline{\nabla}^iT_\delta>V^{\alpha\beta}V^{\gamma\delta}|\nonumber\\
&=2\sum_{\alpha,\gamma}|\lambda_{\alpha}|.|W_{i\alpha}<\nu_\gamma,\overline{\nabla}^iT_\gamma>|.|V_{\alpha}V_{\gamma}|\nonumber\\
&\leq 2\sum_{\alpha,\gamma}|\lambda_{\alpha}|\|W_{\alpha}\|.\|\nu_\gamma\|.\|\overline{\nabla}T_\gamma\|.|V_{\alpha}V_{\gamma}|\nonumber\\
&\leq 4m^{\scriptstyle{\frac{1}{2}}}v^2\|T\|_1\sum_{\alpha,\gamma}|\lambda_{\alpha}|.|V_{\alpha}V_{\gamma}|\nonumber,
\end{align}
where we use $\|W_\beta\|^2\leq v^2$ and  $\|\nu_\gamma\|^2\leq 2mv^2$ from Proposition \ref{p:frameest}.

For each $\alpha$ we use Young's inequality with $a=v\lambda_{\alpha}|V_{\alpha}|$ and $b=2m^{\scriptstyle{\frac{1}{2}}}v\|T\|_1\sum_\gamma |V_{\gamma}|$ to conclude the second estimate
\[
2A_{i\alpha}^jW_{j\beta}<\nu_\gamma,\overline{\nabla}^iT_\delta>V^{\alpha\beta}V^{\gamma\delta}
\leq\epsilon\sum_\alpha v^2\lambda^2_{\alpha}V^2_{\alpha}+4m\epsilon^{-1}\|T\|_1^2v^4
\]
The final term is easily estimated in a similar manner
\[
<\nu_\alpha,\overline{\nabla}_iT_\beta><\nu_\gamma,\overline{\nabla}^iT_\delta>V^{\alpha\beta}V^{\gamma\delta}\leq C_5(m,\|T\|_1)v^4.
\]
Putting these last three estimates together and cancelling the $v^2$ factor we bound the square norm:
\[
|\nabla v|^2\leq(1+\epsilon)\sum_{\alpha}V_\alpha^2\lambda_\alpha^2+C_6(\epsilon,m,\|T\|_1)v^2
\] 
or, rearranging
\begin{equation}\label{e:ineq2a}
\sum_{\alpha}V_\alpha^2\lambda_\alpha^2\geq \frac{1}{1+\epsilon}|\nabla v|^2-C_6v^2.
\end{equation}
Combining inequalities (\ref{e:ineq1}) and (\ref{e:ineq2a}) we obtain
\[
\sum_{\alpha}V_\alpha^2|A_\alpha|^2\geq\left(1+\frac{1}{n}\right)\left[\frac{1}{1+\epsilon}|\nabla
v|^2-C_6v^2\right]-\sum_{\alpha}H_\alpha
^2V_\alpha^2,
\]
which, when substituted in inequality (\ref{e:eq11}), gives
\[
v\left(\frac{d}{ds}-\triangle\right)v{\leq}
-\left(1+\frac{1}{n}\right)\frac{1-\epsilon}{1+\epsilon}|\nabla v|^2+C_7(\epsilon,n,m,|H|,\|T\|_1)v^2+C_4v^4,
\]
and, by virtue of equation (\ref{e:atmax2}),
\[
|\nabla v|^2{=}K^2v^2|\nabla u|^2,
\]
yielding
\begin{equation}\label{e:boxvest}
\left(\frac{d}{ds}-\triangle\right)v{\leq}
-\left(1+\frac{1}{n}\right)\frac{1-\epsilon}{1+\epsilon}K^2v|\nabla u|^2+C_7v+C_4v^3.
\end{equation}
Substituting inequalities (\ref{e:boxuest}) and (\ref{e:boxvest}) in (\ref{e:atmax3}) we get
\[
mKC_1v^2{\geq} \left[\left(1+\frac{1}{n}\right)\frac{1-\epsilon}{(1+\epsilon)}-1\right]K^2|\nabla u|^2-C_7-C_4v^2,
\]
for any $\epsilon>0$.

Now for $0<\epsilon<1/(1+2n)$
\[
\left(1+\frac{1}{n}\right)\frac{1-\epsilon}{1+\epsilon}-1>0,
\]
and so using Proposition \ref{p:slice}
\[
|\nabla u|^2=\sum_{\alpha,\beta}\nabla u_\alpha\cdot\nabla u_\beta \geq{\mbox{min}}_\alpha\psi_\alpha^{-2}(v^2-m),
\]
we have
\[
mKC_1v^2{\geq} C_8(\epsilon,n,|\psi|)K^2(v^2-m)-C_7-C_4v^2,
\]
which can be rearranged to
\[
v^2{\leq}\frac{mC_8K^2+C_7}{C_8K^2-mC_1K-C_4},
\]
where, in summary, $C_1(n,m,\|t\|_2)$, $C_4(\epsilon,n,m,\|T\|_2)$, $C_7(\epsilon,n,m,|H|,\|T\|_1)$ and $C_8(\epsilon,n,|\psi|)$.

For large $K$ this inequality violates $ v\geq m+1$ and we have a contradiction, thereby proving that $h_K<C_K$ in $\Sigma\times [0,s_0)$ and indeed the claim.
\end{proof}
\vspace{0.2in}

\section{Proof of Theorem \ref{t:1}}

For tensors $H_\alpha$ and $A_{ij\alpha}$ we define a positive norm by
\[
|H|_+^2=-H_\alpha H^\alpha \qquad\qquad |A|_+^2=-A_{ij\alpha}A^{ij\alpha},
\]
and similarly for their gradients.

\vspace{0.1in}

\begin{Prop}
Under the mean curvature flow, the norms of the mean curvature vector and the second fundamental form of a spacelike m-dimensional submanifold in an indefinite m+n-dimensional manifold evolve according to:
\[
\left(\frac{d}{ds}-\triangle\right)
|H|_+^2=-2|\tilde{\nabla}H|_+^2-2|H\cdot A|_+^2-2H^\alpha H^\beta\bar{R}_{i\alpha i\beta},
\]
\[
\left(\frac{d}{ds}-\triangle\right)
|A|_+^2=-2|\tilde{\nabla}A|_+^2-2|A|_+^4+A*A*\overline{R}+A*\overline{\nabla}\;\overline{R},
\]
where $\tilde{\nabla}$ is the covariant derivative in both the tangent and normal bundles and $*$ represents linear combinations of contractions
of the tensors involved.
\end{Prop}
\begin{proof}
These are proven in Proposition 4.1 of \cite{LaS}, generalizing the expressions in Proposition 3.3 of \cite{EaH}.
\end{proof}
\vspace{0.1in}

\begin{Prop}\label{p:bdsff}
Under the mean curvature flow
\[
|H|_+^2\leq C_1(1+s^{-1}),
\]
\[
|A|_+^2\leq C_2(1+s^{-1}),
\]
where $C_1=C_1(n,k)$ and $C_2=C_2(n,\|\overline{R}\|_1)$, $k$ being the constant in the timelike curvature condition (\ref{e:tcc}).
\end{Prop}
\begin{proof}
From the previous proposition and the timelike curvature condition
we conclude that
\[
\left(\frac{d}{ds}-\triangle\right) |H|_+^2\leq-2n^{-1}|H|_+^4+2k|H|_+^2,
\]
while
\[
\left(\frac{d}{ds}-\triangle\right) |A|_+^2\leq-2|A|_+^4+C_3|A|_+^2+C_4|A|_+\leq-|A|_+^4+C_5.
\]
The result then follows by a suitable modification of Lemma 4.5 of \cite{EaH}.
\end{proof}

\vspace{0.1in}

We now assemble the proof of Theorem \ref{t:1}:
\vspace{0.1in}

\begin{proof}
The flow is a quasilinear parabolic system and therefore short time existence follows from linear Schauder estimates and the contraction mapping theorem.

Having bounded the gradient and the second fundamental form in Propositions \ref{p:gradest} and \ref{p:bdsff}, bounds on the higher derivatives and 
long-time existence follow from standard parabolic bootstrapping arguments, as in \cite{EaH}.
\end{proof}

\end{document}